\documentclass[11pt,a4paper,reqno]{amsart}
\usepackage[english]{babel}
\usepackage[applemac]{inputenc}
\usepackage[T1]{fontenc}
\usepackage{palatino}
\usepackage{amsmath}
\usepackage{amssymb}
\usepackage{amsthm}
\usepackage{amsfonts}
\usepackage{graphicx}
\usepackage{color}
\usepackage{vmargin}
\usepackage{mathtools}
\usepackage{comment}
\usepackage{amsaddr}

\setmarginsrb{20mm}{20mm}{20mm}{20mm}{10mm}{10mm}{10mm}{10mm}

\usepackage{tikz}
\usepackage{romannum}
\usepackage{pgfplots}
\pgfplotsset{compat=1.15}
\usepackage{mathrsfs}
\usetikzlibrary{arrows}
\usepackage[shortlabels]{enumitem}

\newcommand{\norm}[1]{\left\lVert#1\right\rVert}
\usepackage[final]{pdfpages} 
\let\nic\nearrow
\def\nearrow{\mathbin{{\nic}}}
\setenumerate{leftmargin=*}
\SetEnumerateShortLabel{i}{{\rm(\roman*)},widest=iii}
\SetEnumerateShortLabel{1}{{\rm(\arabic*)}}
\SetEnumerateShortLabel{a}{{\rm(\alph*)},widest=b,align=left}
\SetEnumerateShortLabel{3'}{\rm{(3$^'$)}}

\DeclareMathOperator{\sign}{sign}

\usepackage{fancyhdr}
\fancypagestyle{plain}{
  \fancyhf{}%
  \fancyfoot[LE,RO]{\thepage}%
  
}

\theoremstyle{plain}
\newtheorem{theorem}{Theorem}[section]
\newtheorem{prop}[theorem]{Proposition}
\newtheorem{lemma}[theorem]{Lemma}
\newtheorem{corollary}[theorem]{Corollary}
\theoremstyle{definition}
\newtheorem{definition}{Definition}[section]
\newtheorem{example}{Example}
\newtheorem{remark}{\textnormal{\textbf{Remark}}}[section]
\newtheorem*{remark*}{\textnormal{\textbf{Remark}}}

\makeatother
\begin{document}
\title{Conformal mappings on the Grushin plane}
\author{Marcin Walicki}
\address{Warsaw University of Technology,\\
Faculty of Mathematics and Information Science\\
ul. Koszykowa 75, 00--662, Warsaw, Poland; marcin.walicki.dokt@pw.edu.pl}

\pagenumbering{arabic}
\definecolor{zzttqq}{rgb}{0.6,0.2,0}
\definecolor{xdxdff}{rgb}{0.49019607843137253,0.49019607843137253,1}
\definecolor{ududff}{rgb}{0.30196078431372547,0.30196078431372547,1}
\begin{abstract}
We study conformal mappings in the Grushin plane and provide a number of their characterizations in terms of the Sobolev mappings and their geometry. Furthermore, we connect conformality on the Grushin plane with conformality on the complex plane by using the Meyerson map.  Among applications we discuss admissible curves and length-distortion estimates in the Grushin plane, as well as the Carath\'eodory extension theorem.


\end{abstract}
\maketitle
\section{Introduction}
One of the cornerstones of the geometric analysis is the theory of conformal mappings. In the past century or so the conformal maps have influenced development of quasiconformal mappings and their generalizations, such as quasisymmetric mappings and, more recently, mappings of finite distortion. Moreover, the mapping theory can be studied not only in the Euclidean setting but also beyond it, for instance in Riemannian manifolds and sub-Riemannian setting, see e.g. \cite{Sub-Riemannian geometrybellaiche}, \cite{Cap Cit LD Ott}, \cite{Cowling Ottazzi} as well as in metric measure spaces, see e.g. \cite{Heinonen Koskela}, \cite{Heinonen lectures}. In this note we focus our attention on conformal mappings in the Grushin plane.

The ($\alpha$--)Grushin plane, denoted as $G^2_{\alpha}$, is one of the simplest examples of sub--Riemannian manifolds. Partial differential equations defined on that plane, and more generally on the Grushin spaces, as well as their geometry were subject of recent research, see e.g. \cite{TBieskeJGong}, \cite{TBieske}, \cite{WJM}. The Grushin plane is a modification of the Euclidean plane. One could roughly imagine its structure as the Euclidean plane with singularity being the $y$--axis, which pushes the particles away under the flow given by the horizontal vectors. The Grushin plane posseses the Riemannian metric outside of the singular line, and so, the $y$--axis is the key set, from the perspective of sub--Riemannian geometry.

A new approach to quasisymmetric mappings on the Grushin plane has been introduced in \cite{Ackermann} and later developed for quasiconformal maps in \cite{Quasiconformal mappings on the Grushin plane}. Moreover, \cite{Ackermann} introduces also a new notion of conformal mappings between domains in the Grushin plane. This approach is the starting point for our considerations. Upon recalling and introducing some necessary definitions in Section 2, the following is our main result. We note that $HW_{\textrm{loc}}^{1,p}$ is the so--called \textit{local horizontal Sobolev space} and that sets $S_g$ and $R_g$ stand for \textit{singular set} and \textit{regular set} of $g$ respectively, see Preliminaries for more detailed discussion.

\begin{theorem}[Main theorem]\label{main theorem}
    Let $\Omega, \Omega'$ be open sets in $G^2_{\alpha}$. Let $g=(g_1,g_2)$ be a Sobolev map, $g \in HW^{1,\frac{2(\alpha+1)}{\alpha}}_{\textrm{loc}}(\Omega,\Omega')$. Then the following hold:
    \begin{enumerate}[label=(\alph*)]
        \item If $\overline{W}g = 0$ a.e. in $\Omega$, then $g=\hat{g}$ a.e. in $\Omega$, where the representative $\hat{g}~\in~C^{\infty}(R_{\hat{g}})~\cap~C(\Omega)$.
        \item If $g$ is an orientation preserving bijection, then $\lvert S_g \rvert =0$. In particular, if $g$ is conformal and orientation preserving map, then $\overline{W}g =0$ a.e. in $\Omega$.
        \item If $g$ is a homeomorphism, then $g$ is conformal and orientation preserving if and only if $\overline{W}g=0$ a.e. in $\Omega$ and $\{(x,y) \in \Omega \ | \ g_1(x,y)=0\} = \Omega \cap \{x=0\}$.
        \item If $g$ is a homeomorphism and the number of connected components of set $\Omega \cap \{x=0\}$ is finite, then $g$ is conformal and orientation preserving if and only if $\overline{W}g=0$ a.e. in $\Omega$ and the number of connected components of set $\Omega' \cap \{x=0\}$ is the same as number of connected components of $\Omega \cap \{x=0\}$.
    \end{enumerate}
 
\end{theorem}

Among consequences of the main result we have the following corollaries:
\begin{itemize}
\item[(1)] conformal maps in the Grushin plane can be seen as distorted conformal mappings fixing the $y$--axis, see Theorem \ref{important theorem},
\item[(2)] conformal maps in Grushin plane preserve the admissible curves, see Corollary \ref{preserving admissible curves remark},
\item[(3)] a length-distortion estimate for conformal maps in Grushin plane, see Propostion \ref{length distortion estimate prop},
\item[(4)] a counterpart of the Carath\'eodory extension theorem for Jordan domains in Grushin plane, see Corollary \ref{caratheodory theorem on grushin}.

\end{itemize}
Furthermore, we address the failure of the Riemann mapping theorem on the Grushin plane, see Remark \ref{FoRMT}. Our discussion is illustrated by examples, see e.g. the Joukovski map in Grushin plane (Example \ref{Joukovski map on Grushin}).

\textbf{Acknowledgements:} I would like to deeply thank my academic advisors Tomasz Adamowicz and Krzysztof Che\l mi\'nski for their helpful, factual comments and invaluable help in editorial work. Moreover, I would like to thank Benjamin Warhurst for helpful conversations about conformal and quasi--conformal mappings.

\section{Preliminaries}
The $\alpha$--Grushin plane $G^2_{\alpha}$ is an Euclidean plane $\mathbb{R}^2$ equipped with sub--Riemannian (Carnot--Carath\'eodory) metric on $\mathbb{R}^2$ induced by vector fields
\begin{equation*}
    X = \partial_x, \qquad Y_{\alpha} = \lvert x \rvert^{\alpha} \partial_y, \quad \alpha>0.
\end{equation*}
The Carnot--Carath\'eodory metric is given by the formula
\begin{equation*}
    d_{\alpha}(p,q) := \inf_{\gamma} \int_0^1 \sqrt{\left(\frac{dx}{dt}\right)^2 + \frac{1}{\lvert x \rvert^{2\alpha}}\left(\frac{dy}{dt}\right)^2} \ dt,
\end{equation*}
where the infimum is taken over all curves $\gamma = (x,y):[0,1] \to \mathbb{R}^2$ with $\gamma(0)=p$, $\gamma(1)=q$, absolutely continuous, with respect to the Euclidean metric.

Let $p \ge 1$ and $\Omega \subset G^2_{\alpha}$ be an open set. The \textit{horizontal gradient} of a function $g:\Omega \to \mathbb{R}$ is defined as follows: $\nabla^{\alpha}_H g:=(Xg,Y_{\alpha}g)$. By using this notion, we define \textit{the horizontal Sobolev space} $HW^{1,p,\alpha}(\Omega)$ as a space consisting of such functions $g \in L^p(\Omega)$ whose first weak partial derivatives exist and $\nabla_H^{\alpha}g \in L^p(\Omega)$. Similarly, we define the \textit{horizontal Sobolev space of mappings} $HW^{1,p,\alpha}(\Omega,\Omega')$ as consisting of such maps $g=(g_1,g_2):\Omega \to \Omega'$, for $\Omega$, $\Omega'$ open sets in $G^2_{\alpha}$, that $g_i \in HW^{1,p,\alpha}(\Omega)$ for $i=1,2$. The \textit{local horizontal Sobolev space of mappings}, denoted by $HW^{1,p,\alpha}_{\textrm{loc}}(\Omega,\Omega')$ is defined accordingly, as the space of those maps $g \in HW^{1,p,\alpha}(U,\Omega')$ for every $U \subset \subset \Omega$, where the closure $\overline{U}$ is considered in the Euclidean topology on $\mathbb{R}^2$.

We note that, by definition of weak derivative (cf. Chapter 5 in \cite{Evans} or Chapter 4 in \cite{EvansG}), every function $g \in HW^{1,p,\alpha}_{\textrm{loc}}(\Omega)$ is also in the class $W^{1,1}_{\textrm{loc}}(\Omega)$.

We note that in the above definitions of \textit{local horizontal Sobolev spaces} we used the Lebesgue measure on $\mathbb{R}^2$ to define the $L^p$ spaces. Hence, we will treat a space $HW^{1,p,\alpha}_{\textrm{loc}}(\Omega,\Omega')$ with $\Omega \subset G^2_{\alpha}$ or $\Omega \subset \mathbb{C}$ and $\Omega' \subset G^2_{\alpha}$ or $\Omega' \subset \mathbb{C}$ in the same way.

In what follows, we will usually omit the power $\alpha$ for simplicity of presentation and notation.

Let $g \in HW^{1,1}_{\textrm{loc}}(\Omega)$ for an open set $\Omega \subset G^2_{\alpha}$. Below, when considering vector fields $X g = \frac{\partial g}{\partial x} $, $Y_{\alpha} g = \lvert x \rvert^{\alpha} \frac{\partial g}{\partial y}$, by the $\frac{\partial g}{\partial x} $ and $\frac{\partial g}{\partial y} $ we mean the usual Euclidean partial derivatives, which exist almost everywhere by the ACL characterization of Sobolev functions, see Theorem 4.21 in \cite{EvansG}. In Section 2 the meaning of partial derivatives and weak derivatives is mostly interchangeable, since we consider functions of the class $HW^{1,\frac{2(\alpha + 1)}{\alpha}}_{\textrm{loc}}(\Omega,\Omega')$, cf. Theorem 6.5 in \cite{EvansG}.

Let $M$ be a $C^{\infty}$ Riemannian manifold and $g:M \to M$ be a homeomorphism. Recall that $g$ is conformal if the pullback of the Riemannian metric by $g$ is equal to the metric multiplied by some positive function. Since we assume $M$ is $C^{\infty}$ we also have that $g$ is infinitely differentiable (\cite{Geometrical interpretations of scalar curvature}, page 103, Theorem A). The length element in metric on $G^2_{\alpha} \setminus \{x=0\}$ is 
\begin{equation*}
    dx^2 + \frac{dy^2}{\lvert x \rvert^{2 \alpha}}
\end{equation*}
and its pullback by a mapping $g:G^2_{\alpha} \to G^2_{\alpha}$, $g=(g_1,g_2)$ is defined as follows:
\begin{equation*}
\begin{aligned}
    &\left[(X(g_1))^2+ \frac{(X(g_2))^2}{\lvert g_1 \rvert^{2\alpha}}\right] dx^2 + \frac{1}{\lvert x \rvert^{2\alpha}}\left[(Y_{\alpha}(g_1))^2+ \frac{(Y_{\alpha}(g_2))^2}{\lvert g_1 \rvert^{2\alpha}} \right]dy^2\\
    &+\frac{2}{\lvert x \rvert^{\alpha}}\left[X(g_1)Y_{\alpha}(g_1) + \frac{X(g_2)Y_{\alpha}(g_2)}{\lvert g_1 \rvert^{2\alpha}}\right]dxdy.
\end{aligned}
\end{equation*}
The above calculations are valid in $G^2_{\alpha} \setminus (\{(x,y) \ | \ x=0\}\cup \{(x,y) \ | \ g_1(x,y)=0\})$. This gives us a motivation to define conformal mappings on the Grushin plane.

Based on Definition 12 in \cite{Ackermann} we propose the following definition of conformal mappings in the Grushin plane.
\begin{definition}[Cf. Definition 12 in \cite{Ackermann}]\label{definition of conformal}
    Suppose $\Omega$ and $\Omega'$ are open sets in $G^2_{\alpha}$ and $g=(g_1,g_2)$, $g \in HW^{1,2}_{\textrm{loc}}(\Omega,\Omega')$ is a homeomorphism. Let $S_g = \{(x,y)~\in~\Omega \ | \ x~=~0 \ \textrm{or} \ g_1(x,y)~=~0\}$ denote the \emph{singular set of $g$} and $R_g = \Omega \setminus S_g$, the \emph{regular set of $g$}. We say that $g$ is conformal in $\Omega$, provided that the following two conditions hold:
    \begin{enumerate}
        \item a matrix
        \begin{equation}\label{condition outside singularities}
            D_{\alpha}g = \begin{bmatrix}
X(g_1) & Y_{\alpha}(g_1) \\
\frac{X(g_2)}{\lvert g_1\rvert^{\alpha}} &\frac{Y_{\alpha}(g_2)}{\lvert g_1\rvert^{\alpha}} \\
\end{bmatrix}
        \end{equation}
        is equal almost everywhere, in the sense of the 2--Lebesgue measure, to an orthonormal matrix multiplied by some positive function in $R_g$.
    \item For all points $(x_0,y_0) \in S_g$ the limit
    \begin{equation}\label{limit condition}
        \lim_{R_g \ni (x,y) \to (x_0,y_0)} D_{\alpha} g((x,y))
    \end{equation}
    is defined, finite and non-zero.
    \end{enumerate}
\end{definition}

\begin{remark}
    Notice that the condition \eqref{condition outside singularities} is stated for almost all points since as a map in $HW^{1,2}_{\textrm{loc}}(\Omega, \Omega')$ has partial derivatives defined only up to a set of measure zero. However, our main result, Theorem \ref{main theorem} implies that if $g \in HW^{1,\frac{2(\alpha+1)}{\alpha}}_{\textrm{loc}}(\Omega,\Omega')$ is an orientation--preserving homeomorphism and satisfies \eqref{condition outside singularities} a.e. in $R_g$, then $g \in C^{\infty}(R_g) \cap C(\Omega)$. Therefore, the limit \eqref{limit condition} is well--defined.
\end{remark}

One of the most fundamental tools employed in our work is the so--called Meyerson operator, see \cite{Meyerson}, defined as follows
\begin{equation}\label{Meyerson QS definition}
    \varphi_{\alpha}:G^2_{\alpha} \to \mathbb{C}, \ \varphi_{\alpha}(x,y) = \frac{x \lvert x \rvert^{\alpha}}{\alpha+1} + iy.
\end{equation}
It turns out that $\varphi_{\alpha}$ is quasisymmetric. By using the Meyerson operator $\varphi_{\alpha}$, one can introduce the following differential operators on the Grushin plane, see \cite{Ackermann}:
\begin{equation*}
    Wg = (\partial_x - i\lvert 
 x \rvert^{\alpha} \partial_y)(\varphi_{\alpha} \circ g), \quad \overline{W}g = (\partial_x + i \lvert x \rvert^{\alpha} \partial_y)(\varphi_{\alpha} \circ g),
\end{equation*}
where $g=(g_1,g_2):\Omega \subset G^2_{\alpha} \to G^2_{\alpha}$. Note that $Wg$ and $\overline{W}g$ are counterparts of $\frac{\partial g}{\partial z}$ and $\frac{\partial g}{\partial \bar{z}}$ respectively. Upon direct computations, we obtain the following convenient formulas:
\begin{equation}\label{formulas for wirtinger operators}
\begin{aligned}
        &\overline{W}g = \lvert g_1 \rvert^\alpha \frac{\partial g_1}{\partial x} - \lvert x \rvert^{\alpha} \frac{\partial g_2}{\partial y} + i \left(\frac{\partial g_2}{\partial x} + \lvert x g_1 \rvert^{\alpha} \frac{\partial g_1}{\partial y}\right),\\
       &Wg = \lvert g_1 \rvert^\alpha \frac{\partial g_1}{\partial x} + \lvert x \rvert^{\alpha} \frac{\partial g_2}{\partial y} + i \left(\frac{\partial g_2}{\partial x} - \lvert x g_1 \rvert^{\alpha} \frac{\partial g_1}{\partial y}\right).
\end{aligned}
\end{equation}

The following observation provides a characterization of conformal mappings in the Grushin plane in terms of quasisymmetric mappings. Note that conformality in a sense of Definition \ref{definition of conformal} is a slightly stronger condition than that of Definition 12 in \cite{Ackermann}. First, we recall the definition of an orientation preserving mapping in the Sobolev setting.
\begin{definition}\label{orientation preserving, definition}
    Let $\Omega$, $\Omega'$ be open sets in $G^2_{\alpha}$. A mapping $f: \Omega \to \Omega'$, where $f\in HW^{1,1}_{\textrm{loc}}(\Omega,\Omega')$ is called \emph{orientation preserving} if $J_f > 0$ almost everywhere in $\Omega$, where $J_f$ is the determinant of the matrix $D_H g = \begin{bmatrix}
X(g_1) & Y_{\alpha}(g_1) \\
X(g_2) &Y_{\alpha}(g_2) \\
\end{bmatrix}$. 
\end{definition}
\begin{theorem}[Theorem 4 in \cite{Ackermann}]
    Suppose that $\Omega$, $\Omega'$ are domains in $G^2_{\alpha}$ and $g:\Omega \to \Omega'$ is an orientation-preserving homeomorphism. Then $g$ is conformal on $\Omega \cap R_g$ if and only if $g$ is quasisymmetric and $\overline{W}g=0$ on $R_g$.
\end{theorem}

Before we proceed, we rewrite condition \eqref{condition outside singularities} in Definition \ref{definition of conformal} in a more convenient form.

Recall that orthonormal matrices in $\mathbb{R}^2$ are exactly matrices of rotation by angle $\theta$ or matrices of reflection across a line at the angle of $\frac{\theta}{2}$, i.e.:
\begin{equation*}
    \begin{bmatrix}
\cos{\theta} & -\sin{\theta} \\
\sin{\theta} &\cos{\theta} \\
\end{bmatrix} \ \textrm{(rotation)}, \qquad  
\begin{bmatrix}
\cos{\theta} & \sin{\theta} \\
\sin{\theta} &-\cos{\theta} \\
\end{bmatrix} \ \textrm{(reflection)}.
\end{equation*}

Since we prefer conformal mappings to be orientation preserving, we focus our attention on the case when $g$ is orientation preserving in $R_g$. That is the case when the following condition holds a.e. in $R_g$:
\begin{equation}\label{orientation preserving equivalence}
    \textrm{det}D_{\alpha}g = \frac{X(g_1)Y_{\alpha}(g_2)}{\lvert g_1 \rvert^{\alpha}} - \frac{X(g_2)Y_{\alpha}(g_1)}{\lvert g_1 \rvert^{\alpha}} = \frac{X(g_1)Y_{\alpha}(g_2) - X(g_2)Y_{\alpha}(g_1)}{\lvert g_1 \rvert^{\alpha}} = \frac{J_g}{\lvert g_1 \rvert^{\alpha}} > 0.
\end{equation}
Observe that the above condition rules out the reflections across the line.
Therefore, we are left with the case of $D_{\alpha}g$ being the rotation matrix multiplied by a positive function on $R_g$. Thus, condition \eqref{condition outside singularities} in Definition \ref{definition of conformal} is equivalent to the following
\begin{enumerate}[a]
        \item $\lvert g_1\rvert^{\alpha}X(g_1) = Y_{\alpha}(g_2)$ and $\lvert g_1\rvert^{\alpha} Y_{\alpha}(g_1) = -X(g_2)$ almost everywhere in $R_g$,
        \item $D_{\alpha} g$ is non-singular almost everywhere in $R_g$, i.e. $\textrm{rank}{(D_{\alpha}g)} =2$ a.e. in $R_g$.
\end{enumerate}
Moreover, further investigation of equations \eqref{formulas for wirtinger operators} leads us to the following identity in $R_g$:
\begin{equation}\label{equality of conditions formula}
    \lvert Wg \rvert^2 - \lvert \overline{W}g \rvert^2 = 4\lvert g_1 \rvert^{2\alpha}\textrm{det}D_{\alpha}g.
\end{equation}
Thus, \eqref{condition outside singularities} in Definition \ref{definition of conformal} is equivalent to
\begin{equation}
    \overline{W}g=0, \qquad Wg \neq 0 \ \ \textrm{a.e. in } R_g,
\end{equation}
provided that $g$ is orientation preserving. Indeed, if $\overline{W}g=0$ and $Wg \neq 0$ a.e. in $R_g$, then by \eqref{equality of conditions formula} we obtain that $D_{\alpha}g$ is non--singular a.e. in $R_g$. Conversely, condition (a) implies that $\overline{W}g=0$ a.e. in $R_g$ and since $D_{\alpha}g$ is non--singular, we have by \eqref{equality of conditions formula} that $Wg \neq 0$ a.e. in $R_g$.

\section{Conformal mappings on the Grushin plane}
In this section we will use the following convention notations
\begin{enumerate}
    \item by a measure we mean the Lebesgue measure on $\mathbb{R}^2$ (unless stated otherwise),
    \item a conformal map $g$ between open sets in $G^2_{\alpha}$ is understood in the sense of Definition \ref{definition of conformal}, 
    \item a conformal map $g$ between open sets in $\mathbb{C}$ is understood in the sense of biholomorphism, i.e. $g$ is analytic and is a homeomorphism. In particular $\frac{\partial g}{\partial z} \neq 0$ in the domain of definition of $g$.
\end{enumerate}

Before we proceed recall the following observation.
\begin{lemma}[Theorem 9 in \cite{When is a Function that Satisfies the Cauchy-Riemann Equations Analytic?}]\label{gray morris}
    Let $f=u+iv$ be a locally integrable complex function on an open set $\Omega$. If $f$ satisfies Cauchy--Riemann equations in $\Omega$ in distributional sense, then $f$ has an analytic representative, meaning that $f$ agrees with an analytic function almost everywhere in $\Omega$.
\end{lemma}

We also need some auxiliary results about harmonic functions and solutions to $\overline{W}g=0$.
\begin{lemma}[Estimate on derivatives of harmonic functions]\label{estimates on derivatives of harmonic function}
    Suppose that $u$ is harmonic in an open set $\Omega \subset \mathbb{R}^n$, $r \ge 0$, $z \in \Omega$ and $B(z,r) \subset \Omega$. Then for any multi-index $\alpha$ of order $k \ge 1$
    \begin{equation*}
        \lvert D^{\alpha} u(z) \rvert \le C(n)\left( \frac{2^{n+1}nk}{r} \right)^k \norm{u}_{L^{\infty}(B(z,r))} < k! e^{k}C(n)\left( \frac{2^{n+1}n}{r} \right)^k \norm{u}_{L^{\infty}(B(z,r))}.
    \end{equation*}
\end{lemma}
\begin{proof}
    First part of the inequality can be found in Chapter 2 in \cite{Evans}. The second part follows immediately from the Stirling formula.
\end{proof}
\begin{lemma}\label{wazny lemat}
    Let $u:\Omega \to \Omega'$, $\Omega, \Omega' \subset \mathbb{C}$ be open sets and $u = u_1 + i u_2$  be a conformal mapping such that $u_1(0,y)=0$ for all $(0,y) \in \Omega$. Then 
\begin{equation*}
    \lim_{(x,y) \to (0,y_0)}\frac{u_1(x,y)}{ x } = \frac{\partial u_1}{\partial x}(0,y_0) \neq 0 \textrm{ for all } (0,y_0) \in \Omega.
\end{equation*}
\end{lemma}
\begin{proof}[Proof]
     Since $u$ is conformal we have
    \begin{equation*}
        0 \neq \frac{\partial u}{\partial z}=\frac{\partial u_1}{\partial x} + \frac{\partial u_2}{\partial y} + i\left(\frac{\partial u_2}{\partial x} - \frac{\partial u_1}{\partial y}\right) = 2\frac{\partial u_1}{\partial x} - 2i\frac{\partial u_1}{\partial y},
    \end{equation*}
    and thus, since $\frac{\partial u_1}{\partial y}(0,y_0)=0$ for all $(0,y_0) \in \Omega$ we have that $\frac{\partial u_1}{\partial x}(0,y_0) \neq 0$ for all $(0,y_0) \in \Omega$. Since $u$ is holomorphic, $u_1$ is harmonic and thus real analytic. Expanding $u_1$ into Taylor series in Euclidean ball $B((0,y_0),r) \subset \Omega$, we obtain
    \begin{equation*}
    \begin{aligned}
        u_1(x,y) &= x \frac{\partial u_1}{\partial x}(0,y_0) + \frac{1}{2}\left(x^2 \frac{\partial^2 u_1}{\partial x^2} + 2 x (y-y_0) \frac{\partial^2 u_1}{\partial x \partial y} \right) \\
        &+ \frac{1}{6}\left(x^3\frac{\partial^3 u_1}{\partial x^3} + 3 x^2 (y-y_0) \frac{\partial^3 u_1}{\partial x^2 \partial y} + 3 x (y-y_0)^2 \frac{\partial^3 u_1}{\partial x \partial y^2} \right) + \ldots\\
        &= \sum_{n=1}^{\infty}\left[\frac{1}{n!}\sum_{j=0}^{n-1}{n \choose j} x^{n-j}(y-y_0)^j\frac{\partial^n u}{\partial x^{n-j} \partial y^j}(0,y_0)\right] .
    \end{aligned}
    \end{equation*}
    Therefore,
    \begin{equation}\label{equality that gives lemma}
    \begin{aligned}
        \frac{u_1(x,y)}{x} &= \sum_{n=1}^{\infty}\left[\frac{1}{n!}\sum_{j=0}^{n-1}{n \choose j} x^{n-j-1}(y-y_0)^j\frac{\partial^n u}{\partial x^{n-j} \partial y^j}(0,y_0)\right] \\
        &= \frac{\partial u_1}{\partial x}(0,y_0) + \sum_{n=2}^{\infty}\left[\frac{1}{n!}\sum_{j=0}^{n-1}{n \choose j} x^{n-j-1}(y-y_0)^j\frac{\partial^n u}{\partial x^{n-j} \partial y^j}(0,y_0)\right].
    \end{aligned}
    \end{equation}
Denote by $M:=\norm{u}_{L^{\infty}(B((0,y_0),r))}$, $\rho := \max{\{\lvert x \rvert,\lvert y-y_0\rvert\}}$, $\widetilde{\rho} := \rho^{\frac{n-1}{n}}$ and observe that $\widetilde{\rho} \to 0$ if and only if $(x,y) \to (0,y_0)$. Assume $\widetilde{\rho}<\frac{r}{32e}$. By virtue of Lemma \ref{estimates on derivatives of harmonic function} and the binomial theorem we get the following estimate
\begin{equation*}\label{equality that gives lemma 2}
\begin{aligned}
    &\left\lvert \sum_{n=2}^{\infty}\left[\frac{1}{n!}\sum_{j=0}^{n-1}{n \choose j} x^{n-j-1}(y-y_0)^j\frac{\partial^n u}{\partial x^{n-j} \partial y^j}(0,y_0)\right] \right\rvert \\
    &\le \sum_{n=2}^{\infty}\left[\frac{1}{n!}\sum_{j=0}^{n-1}{n \choose j}  \left\lvert x^{n-j-1}(y-y_0)^j\frac{\partial^n u}{\partial x^{n-j} \partial y^j}(0,y_0) \right\rvert \right] \\
    &\le \sum_{n=2}^{\infty} \rho^{n-1}n!e^{n}C(2)\left(\frac{16}{r}\right)^n M\left[\frac{1}{n!}\sum_{j=0}^{n-1}{n \choose j} \right] \\
    &=\sum_{n=2}^{\infty} \rho^{n-1}C(2)M\left(\frac{16e}{r}\right)^n(2^n-1) \le C(2)M\sum_{n=2}^{\infty}\left(\frac{32e\widetilde{\rho}}{r}\right)^n =C(2)M \frac{(\frac{32e\widetilde{\rho}}{r})^2}{1-\frac{32e\widetilde{\rho}}{r}} \to 0, 
\end{aligned}
\end{equation*}
as $\widetilde{\rho} \to 0 \Leftrightarrow (x,y) \to (0,y_0)$. This together with equality \eqref{equality that gives lemma} imply that
\begin{equation*}
    \frac{u_1(x,y)}{x} \to \frac{\partial u_1}{\partial x}(0,y_0), \quad \textrm{as $(x,y) \to (0,y_0)$},
\end{equation*}
and the assertion of the lemma is proven.
\end{proof}
\begin{lemma}\label{higher integration}
    Let $\Omega$ be an open set in $G^2_{\alpha}$ and function $g \in HW^{1,1}_{\textrm{loc}}(\Omega)$ satisfy almost everywhere in $\Omega$ the following equation
    \begin{equation}\label{satisfy equation}
        \lvert x g \rvert^{\alpha}\frac{\partial g}{\partial y} = f, \textrm{ where }f \in L^1_{\textrm{loc}}(\Omega).
    \end{equation}
     Then $\lvert x \rvert^{\alpha} \lvert g \rvert^{\alpha+1} \in L^1_{\textrm{loc}}(\Omega)$. We emphasize that the above derivative $\frac{\partial g}{\partial y}$ is the usual partial derivative defined almost everywhere, and not the weak derivative of $g$.
\end{lemma}
\begin{proof}    

    Observe that for every rectangle $P = [a,b] \times [c,d] \subset \Omega$, $g$ is an absolutely continuous function on almost every vertical cut of $P$. Indeed, let us consider two cases. First, if $P \cap \{x=0\}=\emptyset$ then this follows from the characterization of Sobolev functions (Theorem 4.21 in \cite{EvansG}).

    If $P \cap \{x=0\} \neq \emptyset$ then this follows from the fact that $g$ is absolutely continuous on almost every vertical cut of rectangle $R \subset P$ such that $R \cap \{x=0\} = \emptyset$ and that countable union of measure--zero sets is a measure--zero set.

It can easily be checked that if $g$ is absolutely continuous on some cut, then the same holds for $\lvert x \rvert^{\alpha} \lvert g \rvert^{\alpha} g$. Therefore, $\lvert x \rvert^{\alpha} \lvert g \rvert^{\alpha} g$ is weakly differentiable on the before--mentioned cut and the corresponding weak derivative is integrable on the cut.

    Fix rectangle $P = [a,b] \times [c,d] \subset \Omega$. By differentiating with respect to $y$ on almost every vertical cut of $P$ we obtain at points where $g \neq 0$ the following equation
    \begin{equation*}
        \frac{\partial}{\partial y}  \left(\lvert x \rvert^{\alpha} \lvert g \rvert^{\alpha} g \right)= \alpha \lvert x \rvert^{\alpha} \lvert g \rvert^{\alpha-1} \frac{g}{\lvert g \rvert} \frac{\partial g}{\partial y}g + \lvert x \rvert^{\alpha} \lvert g \rvert^{\alpha}\frac{\partial g}{\partial y} = (\alpha+1)\lvert x g \rvert^{\alpha} \frac{\partial g}{\partial y}.
    \end{equation*}
The weak derivative is equal to $0$ almost everywhere on the set $\{g=0\}$ by standard Sobolev theory. Hence, for almost every vertical cut of $P$ and every $\varphi \in C_0^{\infty}(P)$ we obtain
\begin{equation*}
    \begin{aligned}
        \int_c^d \lvert x \rvert^{\alpha} \lvert g \rvert^{\alpha} g \frac{\partial \varphi}{\partial y} \ dy = -(\alpha+1)\int_c^d\lvert x g \rvert^{\alpha} \frac{\partial g}{\partial y} \varphi \ dy.
    \end{aligned}
\end{equation*}
Hence, by \eqref{satisfy equation} we get
    \begin{equation*}
    \begin{aligned}
        \int_a^b\int_c^d \lvert x \rvert^{\alpha} \lvert g \rvert^{\alpha} g \frac{\partial \varphi}{\partial y} \ dydx &= -(\alpha+1)\int_a^b\int_c^d\lvert x g \rvert^{\alpha} \frac{\partial g}{\partial y} \varphi \ dydx \\
        &= -(\alpha+1)\iint_P f \varphi \ dxdy < \infty.
    \end{aligned}
\end{equation*}
Finally, we obtain the assertion of the lemma by taking the following test function $\widetilde{\varphi} := y \varphi$, where $\varphi \in C^{\infty}_0(\Omega)$ is such that $\varphi \equiv 1$ on some open set $U \subset \Omega$ containing $(x_0,y_0)$.
\end{proof}
Before we state and prove the main theorem let us make the following observation.
\begin{remark}\label{zero sets to zero sets}
    The Meyerson quasisymmetry $\varphi_{\alpha}$ and its inverse $\varphi_{\alpha}^{-1}$ satisfy the Lusin (N) condition, i.e. they map the Lebesgue measure--zero sets to the Lebesgue measure--zero sets. Indeed, as $\varphi_{\alpha}$ is smooth, it is in particular locally Lipschitz and hence, $\varphi_{\alpha}$ satisfies Lusin (N) condition. The case of $\varphi_{\alpha}^{-1}$ is slightly more complicated. Let $A \subset \mathbb{C}$ be of measure zero. We have
    \begin{equation*}
        \varphi_{\alpha}^{-1}(A) = \varphi_{\alpha}^{-1}(A \cap \{x=0\}) \cup \varphi_{\alpha}^{-1}(A \cap \{x \neq 0\})
    \end{equation*}
    which is Lebesgue measure zero since $\varphi_{\alpha}^{-1}$ maps singular line $\{x=0\}$ on itself and is locally Lipschitz on $\mathbb{C} \setminus \{x=0\}$, cf. formula \eqref{Meyerson QS definition}.

    Another fact, to which we will appeal in the proof of the main theorem, is that $\varphi_{\alpha}$ and $\varphi_{\alpha}^{-1}$ map Lebesgue measurable sets on Lebesgue measurable sets, which also follows from the before--mentioned properties in this remark. That, together with continuity of $\varphi_{\alpha}$ and its inverse, implies that for every measurable mapping $g$, also the following maps
    \begin{equation*}
        \varphi_{\alpha} \circ g, \quad \varphi_{\alpha}^{-1} \circ g, \quad 
         g \circ \varphi_{\alpha}, \quad g \circ \varphi_{\alpha}^{-1}
    \end{equation*}
    are measurable.
\end{remark}

Now we are ready to prove the main result of this work.

\begin{proof}[Proof of Theorem \ref{main theorem}]
We will first show \textbf{assertion (b)}. Let us assume that $g$ is orientation preserving in the sense of Definition \ref{orientation preserving, definition}. Observe that
\begin{equation*}
    \{g^{-1}(0,y) \ | \ (0,y) \in \Omega'\} = \{g^{-1}(0,y) \ | \ (0,y) \in g(\Omega \setminus \{x=0\})\} \cup \{g^{-1}(0,y) \ | \ (0,y) \in g(\{x=0\} \cap \Omega)\}
\end{equation*}
and that $\{g^{-1}(0,y) \ | \ (0,y) \in g(\{x=0\} \cap \Omega)\} \subset \{x=0\}$. Therefore, it is sufficient to show that
\begin{equation*}
    \lvert \{g^{-1}(0,y) \ | \ (0,y) \in g(\Omega \setminus \{x=0\})\} \rvert =0.
\end{equation*}
Fix any open set $U \subset \subset \Omega \setminus \{x=0\}$. By following the reasoning in Step 4 of the proof of Theorem 4.13 in \cite{mappingsoffinitedistortionKoskela} (page 77) one shows that $g|_{U}$ satisfies the Lusin ($N^{-1}$) condition. Indeed, as in \cite{mappingsoffinitedistortionKoskela} let $E \subset U$ be such that $\lvert g(E) \rvert = 0$ and find a Borel set $A \subset \mathbb{R}^2$ with the properties that: $\lvert A \rvert=0$ and $g(E) \subset A$. Then $E$ is contained in the measurable set $E' \subset g^{-1}(A)$. By the change of variables formula (Theorem A.35. in \cite{mappingsoffinitedistortionKoskela}) we have 
\begin{equation*}
    \iint_{E'} \frac{J_g(x,y)}{\lvert x \rvert^{\alpha}} \ dxdy = \iint_{\Omega} \chi_A(g(x,y)) \frac{J_g(x,y)}{\lvert x \rvert^{\alpha}} \ dxdy \le \iint_{\mathbb{R}^2} \chi_A(x,y) \ dxdy=0.
\end{equation*}
where $\chi_A$ stands for the characteristic function of $A$. Note that $\frac{J_g(x,y)}{\lvert x \rvert^{\alpha}}$ is the Euclidean Jacobian, cf. Definition \ref{orientation preserving, definition}. Since $J_g > 0$ a.e. in $\Omega$, it follows that $\lvert E \rvert=0$. We conclude that $\lvert \{g^{-1}(0,y) \ | \ (0,y) \in g(U)\} \rvert =0$. Let $\{B\}_{i=0}^{\infty}$ be a family of open balls such that $\bigcup_{i=0}^{\infty}B_i = \Omega \setminus \{x=0\}$ and $B_i \subset \subset \Omega \setminus \{x=0\}$ for all $i=0,1,\ldots$. By the subadditivity of the Lebesgue measure, we obtain 
\begin{equation*}
\begin{aligned}
    \left\lvert \{g^{-1}(0,y) \ | \ (0,y) \in g(\Omega \setminus \{x=0\})\} \right\rvert &= \left\lvert \bigcup_{i=0}^{\infty} \{g^{-1}(0,y) \ | \ (0,y) \in g(B_i)\} \right\rvert\\
    &\le \sum_{i=0}^{\infty} \lvert \{g^{-1}(0,y) \ | \ (0,y) \in g(B_i)\} \rvert = 0.
\end{aligned}
\end{equation*}

Therefore $\lvert S_g \rvert=0$. In particular, if $g$ is conformal and orientation preserving, then $\overline{W}g = 0$ a.e. in $\Omega$.

Next, we show \textbf{assertion (a)}.

    Let $g:\Omega \to \Omega'$ satisfy the assumptions of the theorem and $\overline{W}g =0$ a.e. in $\Omega$.
    
    Define
    \begin{equation*}
        \widetilde{g} = \varphi_{\alpha} \circ g \circ \varphi_{\alpha}^{-1} : \varphi_{\alpha}(\Omega) \to \varphi_{\alpha}(\Omega'),
    \end{equation*}
where $\varphi_{\alpha}$ stands for the Meyerson quasisymmetry, cf. \eqref{Meyerson QS definition}. Hence, $\widetilde{g}_1$ reads:
\begin{equation}\label{formula for g wave}
    \widetilde{g}_1(x+iy) = \frac{g_1 \lvert g_1 \rvert^{\alpha}}{\alpha+1}\left(\sign{(x)}[(\alpha+1)\lvert x \rvert]^{\frac{1}{\alpha+1}},y\right), \quad (x,y) \in \varphi_{\alpha}(\Omega)
\end{equation}
and
\begin{equation}
        \widetilde{g}_2(x+iy) = g_2\left(\sign{(x)}[(\alpha+1)\lvert x \rvert]^{\frac{1}{\alpha+1}},y\right),
\end{equation}
where $\widetilde{g} = \widetilde{g_1} + i\widetilde{g_2}$. Denote $\widetilde{x} = \sign{(x)}[(\alpha+1)\lvert x \rvert]^{\frac{1}{\alpha+1}}$. We show that $\widetilde{g} \in L^1_{\textrm{loc}}(\varphi(\Omega)),\varphi(\Omega'))$. By $\overline{W}g=0$ a.e. in $\Omega$ and $g \in HW^{1,\frac{2(\alpha+1)}{\alpha}}(\Omega,\Omega')$ we have that $-\frac{\partial g_2}{\partial x} = \lvert x g_1 \rvert^{\alpha} \frac{\partial g_1}{\partial y} \in L^1_{\textrm{loc}}(\Omega)$, since partial derivatives of $g$ and weak partial derivatives of $g$ agree due to $g \in HW^{1,\frac{2(\alpha+1)}{\alpha}}_{\textrm{loc}}(\Omega, \Omega')$. By Lemma \ref{higher integration} we obtain $\lvert x \rvert^{\alpha} \lvert g_1\rvert^{\alpha+1}\in L^{1}_{\textrm{loc}}(\Omega)$. Fix $V \subset \subset \varphi_{\alpha}(\Omega)$ and let $V_+ = V \cap \{x>0\}$, $V_- = V \cap \{x<0\}$. By the change of variables formula applied on $V_+$ we get that
\begin{equation*}
\begin{aligned}
    \iint_{V_+}(\left\lvert \widetilde{g}_1 \right\rvert + \left\lvert \widetilde{g}_2\right\rvert) \ dxdy &= 
    \iint_{V_+} \left\lvert \frac{g_1^{\alpha+1}(\widetilde{x},y)}{\alpha+1} \right\rvert + \lvert g_2(\widetilde{x},y) \rvert \ dxdy\\
    &= \iint_{\varphi_{\alpha}^{-1}(V_+)} \lvert x \rvert^{\alpha}\left(\left\lvert \frac{g_1^{\alpha+1}(x,y)}{\alpha+1} \right\rvert + \lvert g_2(x,y) \rvert\right) \ dxdy<+\infty,
\end{aligned}
\end{equation*}
and thus $\widetilde{g} \in L^1_{\textrm{loc}}(\varphi_{\alpha}(\Omega))$. Similar estimate holds on $V_{-}$. We note that, by Remark \ref{zero sets to zero sets}, map $\widetilde{g}$ is measurable.

We observe that $g \circ \varphi_{\alpha}^{-1} \in HW^{1,1}_{\textrm{loc}}(\varphi_{\alpha}(\Omega),\Omega')$. Indeed, for all $U \subset \subset \varphi_{\alpha}(\Omega)$ we have
\begin{equation*}
    \iint_U \lvert g(\widetilde{x},y) \rvert \ dxdy = \iint_{\varphi_{\alpha}^{-1}(U)} \lvert x \rvert^{\alpha} \lvert g(x,y) \rvert \ dxdy < +\infty
\end{equation*}
and for all $\psi \in C^{\infty}_0(U)$
\begin{equation*}
\begin{aligned}
    \iint_U  g_1(\widetilde{x},y) \frac{\partial \psi}{\partial x}(x,y) \ dxdy &= \iint_{\varphi^{-1}_{\alpha}(U)} \lvert x \rvert^{\alpha} g_1(x,y) \frac{\partial \psi}{\partial x} \left(\frac{x\lvert x\rvert^{\alpha}}{\alpha+1},y\right) \ dxdy\\
    &= \iint_{\varphi^{-1}_{\alpha}(U)} g_1(x,y) \frac{\partial}{\partial x} \left[\psi \left(\frac{x\lvert x\rvert^{\alpha}}{\alpha+1},y\right)\right] \ dxdy\\
    &= - \iint_{\varphi^{-1}_{\alpha}(U)} \frac{\partial g_1}{\partial x}(x,y) \psi \left(\frac{x\lvert x\rvert^{\alpha}}{\alpha+1},y\right) \ dxdy\\
    &=- \iint_U \frac{1}{\lvert \widetilde{x} \rvert^{\alpha}}\frac{\partial g_1}{\partial x}(\widetilde{x},y) \psi(x,y) \ dxdy.
\end{aligned}
\end{equation*}
Moreover
\begin{equation*}
    \iint_U \frac{1}{\lvert \widetilde{x} \rvert^{\alpha}} \left\lvert\frac{\partial g_1}{\partial x}(\widetilde{x},y) \right\rvert \ dxdy = \iint_{\varphi_{\alpha}^{-1}(U)}\left\lvert\frac{\partial g_1}{\partial x}(x,y) \right\rvert \ dxdy <+\infty.
\end{equation*}
Therefore, $\frac{\partial}{\partial x}\left[ g_1(\widetilde{x},y)\right] = \frac{1}{\lvert \widetilde{x} \rvert^{\alpha}}\frac{\partial g_1}{\partial x}(\widetilde{x},y)$ in the weak sense.
Moreover, the existence of all other weak partial derivatives of $g \circ \varphi_{\alpha}^{-1}$ can be proven in the same, straightforward, way. In particular, we obtain that $g \circ \varphi_{\alpha}^{-1}$ is ACL, and therefore $\widetilde{g}$ is ACL as well. By direct computations, it holds that in a set $\{(x,y) \in \Omega \ | \ x \neq 0\}$ we have:
\begin{equation*}
\begin{aligned}
    \frac{\partial \widetilde{g}}{\partial x}(x,y) = \frac{\partial}{\partial x} \left(\frac{g_1 \lvert g_1 \rvert^{\alpha+1}(\widetilde{x},y)}{\alpha+1}\right) = \lvert g_1 (\widetilde{x},y) \rvert^{\alpha}\frac{\partial g_1}{\partial x}(\widetilde{x},y) \frac{\partial \widetilde{x}}{\partial x} = \frac{\lvert g_1(\widetilde{x},y) \rvert^{\alpha}}{\lvert \widetilde{x} \rvert^{\alpha}} \frac{\partial g_1}{\partial x}(\widetilde{x},y).
\end{aligned}
\end{equation*}
Similarly we compute formulas for all other first order partial derivatives of $\widetilde{g}$ and obtain
\begin{equation}\label{usual pd of wg}
    \begin{aligned}
        &\frac{\partial \widetilde{g}_1}{\partial x} = \frac{\lvert g_1 \rvert^{\alpha}}{\lvert \widetilde{x} \rvert^{\alpha}} \frac{\partial g_1}{\partial x}(\widetilde{x},y),\qquad \frac{\partial \widetilde{g}_1}{\partial y}= \lvert g_1 \rvert^{\alpha} \frac{\partial g_1}{\partial y}(\widetilde{x},y),\\
        &\frac{\partial \widetilde{g}_2}{\partial x} = \frac{1}{\lvert \widetilde{x} \rvert^{\alpha}} \frac{\partial g_2}{\partial x}(\widetilde{x},y),\qquad \frac{\partial \widetilde{g}_2}{\partial y}= \frac{\partial g_2}{\partial y}(\widetilde{x},y).
    \end{aligned}
\end{equation}
Thus, we obtain the following:
\begin{equation}\label{this implies bounded gradient}
\begin{aligned}
    \frac{\partial \widetilde{g}}{\partial \bar{z}} &= \frac{\partial \widetilde{g}_1}{\partial x} - \frac{\partial \widetilde{g}_2}{\partial y} + i\left(\frac{\partial \widetilde{g}_2}{\partial x}+\frac{\partial \widetilde{g}_1}{\partial y}\right) = \frac{\lvert g_1 \rvert^{\alpha}}{\lvert \widetilde{x} \rvert^{\alpha}} \frac{\partial g_1}{\partial x}(\widetilde{x},y) - \frac{\partial g_2}{\partial y}(\widetilde{x},y)+ \\ &i\left(\frac{1}{\lvert \widetilde{x} \rvert^{\alpha}} \frac{\partial g_2}{\partial x}(\widetilde{x},y)+ \lvert g_1 \rvert^{\alpha} \frac{\partial g_1}{\partial y}(\widetilde{x},y)\right) = \frac{1}{\lvert \widetilde{x} \rvert^{\alpha}} \overline{W} g (\widetilde{x},y) = 0
\end{aligned}
\end{equation}almost everywhere in $\Omega$. Moreover, by using \eqref{this implies bounded gradient} it is easy to check that partial derivatives of $\widetilde{g}$ defined by equations \eqref{usual pd of wg} are of the class $L^1_{\textrm{loc}}(\varphi_{\alpha}(\Omega))$. Therefore, by the ACL characterization of Sobolev spaces, the partial derivatives \eqref{usual pd of wg} are also weak derivatives of $\widetilde{g}$. From Lemma \ref{gray morris} we may redefine $\widetilde{g}$ on a zero measure set obtaining $\hat{g} = \widetilde{g}$ a.e. so that $\hat{g}$ is holomorphic in $\varphi_{\alpha}(\Omega)$. Therefore, since $\varphi_{\alpha}$ is smooth and $\varphi_{\alpha}^{-1}$ is smooth outside of set $\{(x,y) \in \mathbb{R}^2 \ | \ x=0\}$, the assertion (a) follows.

Let us now proceed to the proofs of \textbf{assertions (c) and (d)}. Suppose that $g$ is homeomorphism and $\overline{W}g=0$ a.e. in $\Omega$. This, by \eqref{this implies bounded gradient} implies that $\widetilde{g}$ is conformal as a complex function.

We will show that $S_g=\{x=0\}$ using \textbf{only} assumptions that $g\in HW^{1,\frac{2(\alpha+1)}{\alpha}}(\Omega,\Omega')$, $g$ being a homeomorphism and $\overline{W}g =0$ a.e. in $\Omega$. Set $\widetilde{x} = \frac{x \lvert x \rvert^{\alpha}}{\alpha+1}$. We solve equations \eqref{usual pd of wg} for the first order partial derivatives of $g$ and obtain
\begin{equation}\label{derivatives of g}
    \begin{aligned}
        &\frac{\partial g_1}{\partial x} = \frac{\lvert x \rvert^{\alpha}}{\lvert g_1 \rvert^{\alpha}}\frac{\partial \widetilde{g}_1}{\partial x}(\widetilde{x},y),\qquad \frac{\partial g_1}{\partial y} = \frac{1}{\lvert g_1 \rvert^{\alpha}}\frac{\partial \widetilde{g}_1}{\partial y}(\widetilde{x},y),\\
        &\frac{\partial g_2}{\partial x} = \lvert x \rvert^{\alpha}\frac{\partial \widetilde{g}_2}{\partial x}(\widetilde{x},y),\qquad \frac{\partial g_2}{\partial y} = \frac{\partial \widetilde{g}_2}{\partial y}(\widetilde{x},y),
    \end{aligned}
\end{equation}
provided that $g_1 \neq 0$. Fix an open ball $B \subset \subset \Omega \setminus \{x=0\}$ and notice that, since $\widetilde{g}_1$ is a harmonic function, by the unique continuation property it cannot vanish on a set of positive measure, unless it is constant. Moreover, $\widetilde{g}_1$ is not constant, since $\widetilde{g}$ is homeomorphism. This observation, together with equality $\overline{W}g=0$ a.e. in $\Omega$, justify the following identity:
\begin{equation}\label{integral identity}
\begin{aligned}
    \iint_B \lvert \nabla_H g_1 \rvert^{\frac{2(\alpha+1)}{\alpha}} \ dxdy =
    \iint_B \frac{\lvert \nabla_H g_2 \rvert^{\frac{2(\alpha+1)}{\alpha}}}{\lvert g_1 \rvert^{2(\alpha+1)}} \ dxdy.
\end{aligned}
\end{equation}
Notice that $g_2 \in C^{\infty}$ by equations \eqref{derivatives of g} and suppose that there is a point $(x_0,y_0) \in B$ such that $g_1(x_0,y_0) = 0$, $\nabla_H g_2(x_0,y_0) \neq 0$ and denote $z_0=\varphi_{\alpha}(x_0,y_0)$. Since $\widetilde{g}_1$ is harmonic we have for a sufficiently small ball $\widetilde{B} \subset B$ centered at $(x_0,y_0)$ the following estimate

\begin{equation*}
\begin{aligned}
    \iint_{\widetilde{B}} \frac{\lvert \nabla_H g_2 \rvert^{\frac{2(\alpha+1)}{\alpha}}}{\lvert g_1 \rvert^{2(\alpha+1)}} \ dxdy &\ge \inf_{\widetilde{B}}\lvert \nabla_H g_2 \rvert^{\frac{2(\alpha+1)}{\alpha}} \iint_{\widetilde{B}} \frac{1}{\lvert g_1 \rvert^{2(\alpha+1)}} \ dxdy\\ 
    &= C(\lvert \nabla_H g_2 \rvert,\widetilde{B}) \iint_{\varphi_{\alpha}(\widetilde{B})}  \frac{\frac{d}{dx}\left[\sign{(x)}[(\alpha+1)\lvert x \rvert]^{\frac{1}{\alpha+1}}\right]}{\lvert g_1 \circ \varphi_{\alpha}^{-1} \rvert^{2(\alpha+1)}} \ dxdy\\
    &\ge C(\lvert \nabla_H g_2 \rvert,\widetilde{B}) \inf_{\widetilde{B}}\left(\frac{1}{\lvert \varphi_{\alpha}^{-1}(x,y)\rvert^{\alpha}} \right)\iint_{\varphi_{\alpha}(\widetilde{B})}  \frac{1}{\lvert g_1 \circ \varphi_{\alpha}^{-1} \rvert^{2(\alpha+1)}} \ dxdy\\ 
    &\ge C(\lvert \nabla_H g_2 \rvert,\varphi_{\alpha}^{-1},\widetilde{B})\iint_{\varphi_{\alpha}(\widetilde{B})} \frac{1}{\lvert \widetilde{g}_1 \rvert^2} \ dxdy= C\iint_{\varphi_{\alpha}(\widetilde{B})} \frac{1}{\lvert \widetilde{g}_1(x+iy) - \widetilde{g}_1(z_0) \rvert^2} \ dxdy\\
    &\ge \frac{C}{\textrm{Lip}_{\widetilde{g}_1}}\iint_{\varphi_{\alpha}(\widetilde{B})} \frac{1}{\lvert x+iy-z_0 \rvert^2} \ dxdy = \infty
\end{aligned}
\end{equation*}
This, by \eqref{integral identity} contradicts the fact that $\nabla_H g_1 \in L^{\frac{2(\alpha+1)}{\alpha}}(\widetilde{B})$, and therefore $\nabla_H g_2(x_0,y_0)=0$ for $x_0 \neq 0$ if $g_1(x_0,y_0)=0$. Notice that, since $\widetilde{g}$ is conformal, we have that at every point $(\widetilde{x},y)$ either $\frac{\partial \widetilde{g}_i}{\partial x}$ or $\frac{\partial \widetilde{g}_i}{\partial y}$ is non-zero for $i=1,2$ and the argument for this fact is the same as in the beginning of the proof of Lemma \ref{wazny lemat}. Thus, by using equations \eqref{derivatives of g}, observation that formulas for $\frac{\partial g_2}{\partial x}$, $\frac{\partial g_2}{\partial y}$ are valid everywhere and the fact that $\nabla_H g_2 =0$ if $g_1=0$ and $x \neq 0$ we obtain that $g_1 \neq 0$ outside of $\{x=0\}$. This shows that $S_g = \{x=0\}$.

We will show the first implication in (c), i.e. if $g$ is conformal and orientation preserving then $\overline{W}g=0$ a.e. in $\Omega$ and $\{g_1 = 0\} = \Omega \cap \{x=0\}$. The first part is obvious by (b). We show that $\{g_1 = 0\} = \Omega \cap \{x=0\}$. We already proven that $\{g_1 = 0\} \subset \Omega \cap \{x=0\}$, i.e. $S_g = \{x=0\}$. We will show the reverse inclusion.  Assume on the contrary that $g_1 \neq 0$ at some point $(0,y_0) \in \Omega$. Then, by continuity of $g_1$, it is true in some neighbourhood of $(0,y_0)$ that $g_1 \neq 0$ and by using equations \eqref{derivatives of g} we deduce that $g \in C^{\infty}$ in that neighbourhood. This contradicts the fact that $g$ is conformal, since then the limit
    \begin{equation}\label{limit if smooth}
    \begin{aligned}
       \lim_{(x,y) \to (0,y_0)} D_{\alpha}g (x,y) = \lim_{(x,y) \to (0,y_0)} \begin{bmatrix}
X(g_1) & Y_{\alpha}(g_1) \\
\frac{X(g_2)}{\lvert g_1\rvert^{\alpha}} &\frac{Y_{\alpha}(g_2)}{\lvert g_1\rvert^{\alpha}} \\
\end{bmatrix} = \lim_{(x,y) \to (0,y_0)} \begin{bmatrix}
\frac{Y_{\alpha}(g_2)}{\lvert g_1\rvert^{\alpha}} & Y_{\alpha}(g_1) \\
-Y_{\alpha}(g_1) &\frac{Y_{\alpha}(g_2)}{\lvert g_1\rvert^{\alpha}}\\
\end{bmatrix}=0.
    \end{aligned}
    \end{equation}
Therefore $\{g_1 = 0\} = \Omega \cap \{x=0\}$.

We now prove that the converse implication in (c) is true. Assume that $\{g_1 = 0\} = \Omega \cap \{x=0\}$ and $\overline{W}g=0$ a.e. in $\Omega$. We have in $R_g = \Omega \setminus \{x=0\}$
\begin{equation}\label{orientation preserving}
    \textrm{det}D_{\alpha}g = \frac{\lvert x \rvert^{\alpha}}{\lvert g_1 \rvert^{\alpha}}\frac{\partial g_1}{\partial x} \frac{\partial g_2}{\partial y} - \frac{\lvert x \rvert^{\alpha}}{\lvert g_1 \rvert^{\alpha}}\frac{\partial g_1}{\partial y}\frac{\partial g_2}{\partial x} = \frac{\lvert x \rvert^{2\alpha}}{\lvert g_1 \rvert^{2\alpha}}J_{\widetilde{g}}(\widetilde{x},y) >0,
\end{equation}
and, by the fact that $\widetilde{g}$ is conformal, $D_{\alpha}g$ is non--singular in $\Omega \setminus \{x=0\}$. Thus, we have shown that $g$ satisfies condition \eqref{condition outside singularities} in Definition \ref{definition of conformal}. Next we show \eqref{limit condition} in that definition. Note that outside the vertical line $\{x=0\}$ it holds that
\begin{equation}\label{alternative condition follows}
    D_{\alpha}g = \begin{bmatrix}
X(g_1) & Y_{\alpha}(g_1) \\
\frac{X(g_2)}{\lvert g_1\rvert^{\alpha}} &\frac{Y_{\alpha}(g_2)}{\lvert g_1\rvert^{\alpha}} \\
\end{bmatrix} = \begin{bmatrix}
\frac{\lvert x \rvert^{\alpha}}{\lvert g_1 \rvert^{\alpha}}\frac{\partial \widetilde{g}_1}{\partial x} & \frac{\lvert x \rvert^{\alpha}}{\lvert g_1 \rvert^{\alpha}}\frac{\partial \widetilde{g}_1}{\partial y} \\
\frac{\lvert x \rvert^{\alpha}}{\lvert g_1 \rvert^{\alpha}}\frac{\partial \widetilde{g}_2}{\partial x} &\frac{\lvert x \rvert^{\alpha}}{\lvert g_1 \rvert^{\alpha}}\frac{\partial \widetilde{g}_2}{\partial y} \\
\end{bmatrix} = 
C(\alpha) \frac{\lvert \widetilde{x} \rvert^{\frac{\alpha}{\alpha+1}}}{\lvert \widetilde{g}_1\rvert^{\frac{\alpha}{\alpha+1}}}\begin{bmatrix}
\frac{\partial \widetilde{g}_1}{\partial x} & \frac{\partial \widetilde{g}_1}{\partial y} \\
\frac{\partial \widetilde{g}_2}{\partial x} &\frac{\partial \widetilde{g}_2}{\partial y} \\
\end{bmatrix}.
\end{equation}
Hence, by continuity of derivatives $\widetilde{g}$ and Lemma \ref{wazny lemat}, we conclude that the limit in \eqref{limit if smooth} exists as we approach via paths contained in $\Omega \setminus \{x=0\}$. Moreover, this limit is non-zero since $\frac{\partial \widetilde{g}_1}{\partial x}(0,y) \neq 0$ for all $(0,y) \in \varphi_{\alpha}(\Omega)$. Therefore, condition \eqref{limit condition} in Definition \ref{definition of conformal} holds and $g$ is conformal. We conclude the proof of assertion (c) by observing that equation \eqref{orientation preserving} together with \eqref{orientation preserving equivalence} imply that $g$ is orientation preserving.

Finally, we show \textbf{assertion (d)}. By assumptions we may write $\Omega \cap \{x=0\} = \bigcup_{i=0}^N I_i$, where $I_i$ are connected components of $\Omega \cap \{x=0\}$. Assume that $g$ is conformal and orientation preserving. By (c) we know that $\{g_1=0\}=\Omega \cap \{x=0\}$. Since each of $I_i$ is closed and connected in the underlying topology we have that $g(I_i) \subset J$ is closed and connected for all $i=0,1,\ldots, N$, where $J$ is one of the connected components of $\Omega' \cap \{x=0\}$, homeomorphic to $(0,1)$. Without loss of generality assume that $J=(0,1)$. If either 
\begin{equation*}
    g(I_i) = (0, a], \qquad g(I_i) = [a,b], \qquad g(I_i) = [a,1),
\end{equation*}
for some $a,b$ then $g(I_i) \setminus \{a\}$ is connected but $g^{-1}(g(I_i) \setminus \{a\})$ is not. Therefore $g(I_i) = J_j$ and this shows the sufficiency part of assertion (d).

Assume that $g$ is a homeomorphism, $\overline{W}g=0$ almost everywhere and that the number of connected components of $\Omega' \cap \{x=0\}$ is the same as the number of connected components of $\Omega \cap \{x=0\}$. This implies that $\{g_1 = 0\} = \Omega \cap \{x=0\}$. Indeed, we know by previous parts of the proof that $S_g = \{x=0\}$, and therefore $\Omega' \cap \{x=0\} \subset g(\Omega \cap \{x=0\})$. Let $\Omega' \cap \{x=0\} = \bigcup_{i=0}^N J_i$ and observe that each of $J_i$ is closed and connected in the underlying topology. Hence, similarly as above, $g^{-1}(J_i) \subset I_i$, where $I_i$ is one of the connected components of $\Omega \cap \{x=0\}$, and $g^{-1}(J_i)$ is closed and connected for all $i=0,1,\ldots, N$. As above, we obtain that in fact $g^{-1}(J_i) = I_i$ for some $i$. This shows that $\{g_1 = 0\} = \Omega \cap \{x=0\}$ and that completes the proof of assertion (d) and the proof of the whole theorem as well.
\end{proof}

\begin{remark}
    Notice that the assumption $\{g_1 = 0\} = \Omega \cap \{x=0\}$ in Theorem \ref{main theorem} is nessesary. Indeed, let $\widetilde{g}$ be conformal mapping on the complex plane such that $\widetilde{g}_1 \neq 0$ at some point $x_0 + iy_0=i y_0$. It holds that
    \begin{equation*}
        \overline{W}g = \overline{W}(\varphi_{\alpha}^{-1} \circ \widetilde{g} \circ \varphi_{\alpha})=0
    \end{equation*}
    in $R_g$. However, by \eqref{limit if smooth} $g$ is not conformal.
\end{remark}

We remark that Theorem \ref{main theorem} and its proof gives us the following observation.
\begin{theorem}\label{important theorem}
    Let $\Omega, \Omega'$ be open in $G^2_{\alpha}$. The family of orientation preserving, conformal mappings $g~=~(g_1,g_2)$, $g \in HW^{1,\frac{2(\alpha+1)}{\alpha}}_{\textrm{loc}}(\Omega,\Omega')$ is equivalent to the family of conformal maps $\widetilde{g}:\varphi_{\alpha}(\Omega) \to \varphi_{\alpha}(\Omega')$ such that $\widetilde{g}_1(x,y) =0$ if and only if $x=0$.

    Moreover, the equivalence is given by the formula $\widetilde{g} = \varphi_{\alpha} \circ g \circ \varphi_{\alpha}^{-1}$ and $g_1(x,y)=0$ if and only if $x=0$.
\end{theorem}
\begin{proof}[Proof]
The proof goes similar to the one of Theorem \ref{main theorem}, and therefore we present only the sketch of the reasoning.

If $g:\Omega \to \Omega'$ is an orientation preserving conformal mapping then $\widetilde{g}:\varphi_{\alpha}(\Omega) \to \varphi_{\alpha}(\Omega')$ given by the formula $\widetilde{g} = \varphi_{\alpha} \circ g \circ \varphi_{\alpha}^{-1}$ is conformal as well and $\widetilde{g}_1(x,y)=0$ if and only if $x=0$.

For the proof of the opposite implication, fix conformal mapping $\widetilde{g}:\varphi_{\alpha}(\Omega) \to \varphi_{\alpha}(\Omega')$ such that $\widetilde{g}_1(x,y)=0$ if and only if $x=0$.

As in Theorem \ref{main theorem} we define $g = \varphi_{\alpha}^{-1} \circ \widetilde{g} \circ \varphi_{\alpha}$ and by using equations \eqref{derivatives of g} we find that
\begin{equation*}
    (\overline{W}g)(x,y)= \lvert x \rvert^{\alpha} \frac{\partial \widetilde{g}}{\partial \bar{z}}(\widetilde{x},y)=0 \quad \textrm{in } R_g.
\end{equation*}
Thus, $\overline{W}g=0$ in $R_g = \Omega \setminus \{x=0\}$. Using \eqref{formula for g wave} we can easily check that $g_1(x,y)=0$ if and only if $x=0$. What is left to be shown is that $g \in HW^{1,\frac{2(\alpha+1)}{\alpha}}_{\textrm{loc}}(\Omega,\Omega')$. We show that using equations \eqref{derivatives of g} and Lemma \ref{wazny lemat}. The fraction $\frac{\lvert x \rvert^{\alpha}}{\lvert g_1 \rvert^{\alpha}}$ is defined in $R_g = \Omega \setminus \{x=0\}$. We extend it continuously to $\{x=0\}$: 
\begin{equation*}
    \textrm{Ext}(x,y) = \begin{cases}  \frac{\lvert x \rvert^{\alpha}}{\lvert g_1 \rvert^{\alpha}}, \qquad \qquad \quad \ \ \ x \neq 0,\\ 
    \left(\frac{1}{\frac{\partial \widetilde{g}_1}{\partial x}(0,y)}\right)^{\frac{\alpha}{\alpha+1}}, \quad x=0.
    \end{cases}
\end{equation*}
For proof of continuity cf. \eqref{alternative condition follows} and Lemma \ref{wazny lemat}. By using continuity of $\textrm{Ext}:\Omega \to \mathbb{R}$ in \eqref{derivatives of g} we easily get desired result.
\end{proof}

\section{Applications of Theorem \ref{main theorem}}
In this section, we use Theorem \ref{main theorem} and the related observations to derive some further results for conformal mappings on the Grushin plane.

\begin{corollary}[Basic properties of Grushin--conformal mappings]\label{basic properties of conformal}
Let $g~=~(g_1,g_2)$ and $h=(h_1,h_2)$ be orientation preserving conformal mappings between open sets in the Grushin plane such that $g,h \in HW^{1,\frac{2(\alpha+1)}{\alpha}}_{\textrm{loc}}$. Then:
\begin{enumerate}
    \item $g \circ h$ and $h \circ g$ are orientation preserving and conformal,
    \item $g^{-1}$ and $h^{-1}$ are orientation preserving and conformal,
\end{enumerate}
where all the domains and their images are such that the above expressions make sense.

In particular, given a natural dilation on the Grushin plane, $\delta_{\lambda}(x,y) = (\lambda x, \lambda^{\alpha+1}y)$, $\lambda>0$, we have that $g \circ \delta_{\lambda}$ and $\delta_{\lambda} \circ g$ are conformal.
\end{corollary}

\begin{proof} In order to show assertion (1) let $\widetilde{g}$ and $\widetilde{h}$ be defined as in Theorem \ref{important theorem}. Then
    \begin{equation*}
        g \circ h = \varphi_{\alpha}^{-1} \circ \widetilde{g} \circ \varphi_{\alpha} \circ \varphi_{\alpha}^{-1} \circ \widetilde{h} \circ \varphi_{\alpha} = \varphi_{\alpha}^{-1} \circ \widetilde{g} \circ \widetilde{h} \circ \varphi_{\alpha}.
    \end{equation*}
    By Theorem \ref{important theorem} this completes the proof of (1), since $\widetilde{g} \circ \widetilde{h}=0$ if and only if $\{x=0\}$.

    The proofs of remaining assertions go the same lines and therefore we omit them.
\end{proof}

\begin{remark}[The failure of the Riemann mapping theorem in the Grushin plane]\label{FoRMT}
  There exist $\Omega, \Omega'$ simply--connected domains in $G^2_{\alpha}$ such that $\Omega \cap \{x=0\}$ and $\Omega' \cap \{x=0\}$ have the same number of connected components but there exists no orientation preserving conformal mapping $g~\in~ HW^{1,\frac{2(\alpha+1)}{\alpha}}_{\textrm{loc}}(\Omega,\Omega')$.
Indeed, let $\Omega, \Omega' \subset G^2_{\alpha}$ be defined as follows:

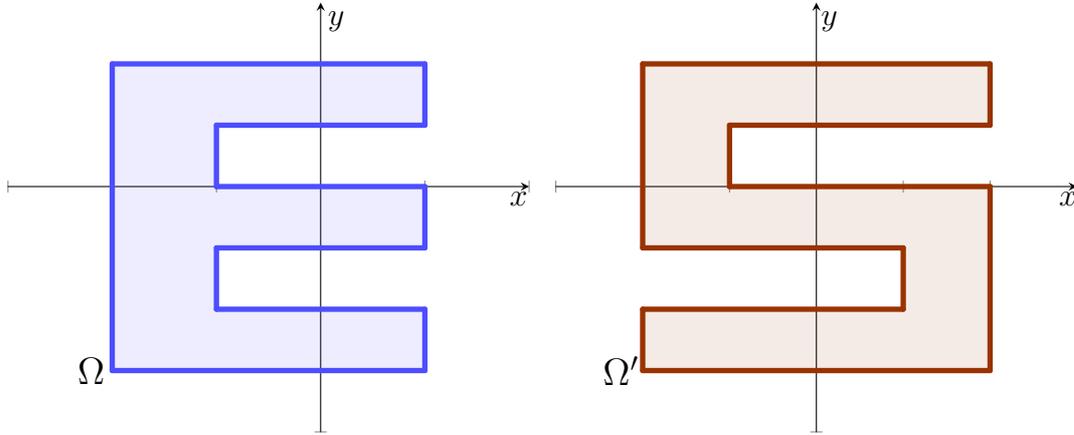
\begin{tikzpicture}[line cap=round,line join=round,>=triangle 45]
\begin{axis}[
axis lines=middle,
ymajorgrids=false,
xmajorgrids=false,
xmin=-3,
xmax=2,
ymin=-4,
ymax=3,
yticklabels={,,},
xticklabels={,,}
]
\clip(-8.589199177310892,-5.647098651559253) rectangle (13.153696529663714,7.4592461474184715);
\fill[line width=2pt,color=ududff,fill=ududff,fill opacity=0.10000000149011612] (-2,2) -- (1,2) -- (1,1) -- (-1,1) -- (-1,0) -- (1,0) -- (1,-1) -- (-1,-1) -- (-1,-2) -- (1,-2) -- (1,-3) -- (-2,-3) -- cycle;
\draw [line width=2pt,color=ududff] (-2,2)-- (1,2);
\draw [line width=2pt,color=ududff] (1,2)-- (1,1);
\draw [line width=2pt,color=ududff] (1,1)-- (-1,1);
\draw [line width=2pt,color=ududff] (-1,1)-- (-1,0);
\draw [line width=2pt,color=ududff] (-1,0)-- (1,0);
\draw [line width=2pt,color=ududff] (1,0)-- (1,-1);
\draw [line width=2pt,color=ududff] (1,-1)-- (-1,-1);
\draw [line width=2pt,color=ududff] (-1,-1)-- (-1,-2);
\draw [line width=2pt,color=ududff] (-1,-2)-- (1,-2);
\draw [line width=2pt,color=ududff] (1,-2)-- (1,-3);
\draw [line width=2pt,color=ududff] (1,-3)-- (-2,-3);
\draw [line width=2pt,color=ududff] (-2,-3)-- (-2,2);
\node[] at (-2.2,-3) {\Large $\Omega$};
\node[] at (1.9,-0.2) {\large $x$};
\node[] at (0.15,2.7) {\large $y$};
\end{axis}
\end{tikzpicture}
\begin{tikzpicture}[line cap=round,line join=round,>=triangle 45,x=1cm,y=1cm]
\begin{axis}[
axis lines=middle,
ymajorgrids=false,
xmajorgrids=false,
xmin=-3,
xmax=3,
ymin=-4,
ymax=3,
yticklabels={,,},
xticklabels={,,}]
\clip(-8.88,-5.98) rectangle (14.08,7.86);
\fill[line width=2pt,color=zzttqq,fill=zzttqq,fill opacity=0.10000000149011612] (-2,2) -- (2,2) -- (2,1) -- (-1,1) -- (-1,0) -- (2,0) -- (2,-3) -- (-2,-3) -- (-2,-2) -- (1,-2) -- (1,-1) -- (-2,-1) -- cycle;
\draw [line width=2pt,color=zzttqq] (-2,2)-- (2,2);
\draw [line width=2pt,color=zzttqq] (2,2)-- (2,1);
\draw [line width=2pt,color=zzttqq] (2,1)-- (-1,1);
\draw [line width=2pt,color=zzttqq] (-1,1)-- (-1,0);
\draw [line width=2pt,color=zzttqq] (-1,0)-- (2,0);
\draw [line width=2pt,color=zzttqq] (2,0)-- (2,-3);
\draw [line width=2pt,color=zzttqq] (2,-3)-- (-2,-3);
\draw [line width=2pt,color=zzttqq] (-2,-3)-- (-2,-2);
\draw [line width=2pt,color=zzttqq] (-2,-2)-- (1,-2);
\draw [line width=2pt,color=zzttqq] (1,-2)-- (1,-1);
\draw [line width=2pt,color=zzttqq] (1,-1)-- (-2,-1);
\draw [line width=2pt,color=zzttqq] (-2,-1)-- (-2,2);
\node[] at (-2.25,-3) {\Large $\Omega'$};
\node[] at (2.9,-0.2) {\large $x$};
\node[] at (0.15,2.7) {\large $y$};
\end{axis}
\end{tikzpicture}

In the picture below, we denote connected components of $\Omega \cap \{x=0\}$ by $a,b,c$, connected components of $\Omega' \cap \{x=0\}$ by $a',b',c'$ and by $A,B,C,D$ the connected components of $\Omega \cap \{x<0\}$, $\Omega \cap \{x>0\}$. Similarly, we denote the corresponding components of set $\Omega'$.

\begin{tikzpicture}[line cap=round,line join=round,>=triangle 45]
\begin{axis}[
axis lines=middle,
ymajorgrids=false,
xmajorgrids=false,
xmin=-3,
xmax=2,
ymin=-4,
ymax=3,
yticklabels={,,},
xticklabels={,,}
]
\clip(-8.589199177310892,-5.647098651559253) rectangle (13.153696529663714,7.4592461474184715);
\fill[line width=2pt,color=ududff,fill=ududff,fill opacity=0.10000000149011612] (-2,2) -- (1,2) -- (1,1) -- (-1,1) -- (-1,0) -- (1,0) -- (1,-1) -- (-1,-1) -- (-1,-2) -- (1,-2) -- (1,-3) -- (-2,-3) -- cycle;
\draw [line width=2pt,color=ududff] (-2,2)-- (1,2);
\draw [line width=2pt,color=ududff] (1,2)-- (1,1);
\draw [line width=2pt,color=ududff] (1,1)-- (-1,1);
\draw [line width=2pt,color=ududff] (-1,1)-- (-1,0);
\draw [line width=2pt,color=ududff] (-1,0)-- (1,0);
\draw [line width=2pt,color=ududff] (1,0)-- (1,-1);
\draw [line width=2pt,color=ududff] (1,-1)-- (-1,-1);
\draw [line width=2pt,color=ududff] (-1,-1)-- (-1,-2);
\draw [line width=2pt,color=ududff] (-1,-2)-- (1,-2);
\draw [line width=2pt,color=ududff] (1,-2)-- (1,-3);
\draw [line width=2pt,color=ududff] (1,-3)-- (-2,-3);
\draw [line width=2pt,color=ududff] (-2,-3)-- (-2,2);
\node[] at (-2.2,-3) {\Large $\Omega$};
\node[] at (-1.5,-0.5) {\large $A$};
\node[] at (0.5,1.5) {\large $B$};
\node[] at (0.5,-0.5) {\large $C$};
\node[] at (0.5,-2.5) {\large $D$};
\node[] at (-0.2,1.5) {\large $a$};
\node[] at (-0.2,-0.5) {\large $b$};
\node[] at (-0.2,-2.5) {\large $c$};
\node[] at (1.9,-0.2) {\large $x$};
\node[] at (0.15,2.7) {\large $y$};
\end{axis}
\end{tikzpicture}
\begin{tikzpicture}[line cap=round,line join=round,>=triangle 45,x=1cm,y=1cm]
\begin{axis}[
axis lines=middle,
ymajorgrids=false,
xmajorgrids=false,
xmin=-3,
xmax=3,
ymin=-4,
ymax=3,
yticklabels={,,},
xticklabels={,,}]
\clip(-8.88,-5.98) rectangle (14.08,7.86);
\fill[line width=2pt,color=zzttqq,fill=zzttqq,fill opacity=0.10000000149011612] (-2,2) -- (2,2) -- (2,1) -- (-1,1) -- (-1,0) -- (2,0) -- (2,-3) -- (-2,-3) -- (-2,-2) -- (1,-2) -- (1,-1) -- (-2,-1) -- cycle;
\draw [line width=2pt,color=zzttqq] (-2,2)-- (2,2);
\draw [line width=2pt,color=zzttqq] (2,2)-- (2,1);
\draw [line width=2pt,color=zzttqq] (2,1)-- (-1,1);
\draw [line width=2pt,color=zzttqq] (-1,1)-- (-1,0);
\draw [line width=2pt,color=zzttqq] (-1,0)-- (2,0);
\draw [line width=2pt,color=zzttqq] (2,0)-- (2,-3);
\draw [line width=2pt,color=zzttqq] (2,-3)-- (-2,-3);
\draw [line width=2pt,color=zzttqq] (-2,-3)-- (-2,-2);
\draw [line width=2pt,color=zzttqq] (-2,-2)-- (1,-2);
\draw [line width=2pt,color=zzttqq] (1,-2)-- (1,-1);
\draw [line width=2pt,color=zzttqq] (1,-1)-- (-2,-1);
\draw [line width=2pt,color=zzttqq] (-2,-1)-- (-2,2);
\node[] at (-2.25,-3) {\Large $\Omega'$};
\node[] at (-1.5,0.5) {\large $A'$};
\node[] at (1,1.5) {\large $B'$};
\node[] at (1.5,-1.5) {\large $C'$};
\node[] at (-1.5,-2.5) {\large $D'$};
\node[] at (-0.2,1.5) {\large $a'$};
\node[] at (-0.2,-0.5) {\large $b'$};
\node[] at (-0.2,-2.5) {\large $c'$};
\node[] at (2.9,-0.2) {\large $x$};
\node[] at (0.15,2.7) {\large $y$};
\end{axis}
\end{tikzpicture}

Assume that $g \in HW^{1,\frac{2(\alpha+1)}{\alpha}}_{\textrm{loc}}(\Omega,\Omega')$ is an orientation preserving conformal map. From the proof of assertion (d) in Theorem \ref{main theorem} we know that connected components of $\Omega \cap \{x=0\}$ are mapped by $g$ on connected components of $\Omega' \cap \{x=0\}$. Let us analyse the possible cases.
\begin{enumerate}
    \item Assume that $g(a) = a'$. Then $\overline{B} \cap \{x=0\}$ has one connected component and $\overline{A'} \cap \{x=0\}$ has two connected components (see the picture). Therefore $g(B)=B'$, and so $g(A)=A'$ but $\overline{A} \cap \{x=0\}$ has three connected components, while $\overline{A'} \cap \{x=0\}$ has only two, which contradicts the conformality of $g$. 
    \item Assume that $g(a) = b'$. Then $\overline{C} \cap \{x=0\}$ has one connected component, while both, $\overline{A'} \cap \{x=0\}$ and $\overline{C'} \cap \{x=0\}$ have two connected components. Again, this may not happen for conformal $g$.
    \item Assume $g(a) = c'$. This case is analogous to case (1).
\end{enumerate}
In conclusion, an orientation preserving $g \in HW^{1,\frac{2(\alpha+1)}{\alpha}}_{\textrm{loc}}(\Omega,\Omega')$ cannot be conformal.
\end{remark}

Recall that an absolutely continuous curve $\gamma:[a,b] \to G^2_{\alpha}$ is admissible if and only if there exist $u_1,u_2 \in L^1([a,b])$ such that
    \begin{equation*}
        \gamma'(t) = u_1(t)X(\gamma(t)) + u_2(t)Y_{\alpha}(\gamma(t)), \quad \textrm{ for a.e. } t \in [a,b],
    \end{equation*}
    cf. \cite{Sub-Riemannian geometrybellaiche} page 6. That is equivalent to $\gamma_1', \frac{\gamma_2'}{\lvert \gamma_1\rvert^{\alpha}} \in L^1([a,b])$.
\begin{corollary}[Preservation of admissible curves]\label{preserving admissible curves remark}
    Let $\Omega, \Omega' \subset G^2_{\alpha}$ be open and assume that $g \in HW^{1,\frac{2(\alpha+1)}{\alpha}}(\Omega,\Omega')$ is an orientation preserving conformal map. Fix an absolutely continuous (in the Euclidean metric) curve $\gamma~:~[a,b]\to~\Omega~$ s.t. $\gamma_1 \neq 0$ a.e. in $[a,b]$. Then, $\gamma$ is admissible if and only if $g(\gamma)$ is admissible.

\end{corollary}

\begin{proof}
    Let $\gamma:[a,b]\to G^2_{\alpha}$ be an absolutely continuous curve. Recalling that $\gamma_1^{-1}(0)$ is closed by continuity of $\gamma$ and using the property $\overline{W}g=0$ in $R_g$, see \eqref{formulas for wirtinger operators}, we compute:
    \begin{equation}\label{pres of adm curves}
    \begin{aligned}
    &\frac{d}{d t} g_1(\gamma(t)) = \frac{\partial g_1}{\partial x} \gamma_1' +  \frac{\partial g_1}{\partial y} \gamma_2' = \frac{\partial g_1}{\partial x} \gamma_1' + \frac{\gamma_2'}{\lvert \gamma_1 \rvert^{\alpha}} \lvert \gamma_1 \rvert^{\alpha} \frac{\partial g_1}{\partial y}\\
    &\frac{d}{d t} g_2(\gamma(t)) = \frac{\partial g_2}{\partial x} \gamma_1' +  \frac{\partial g_2}{\partial y} \gamma_2' = - \lvert \gamma_1 g_1 \rvert^{\alpha} \frac{\partial g_1}{\partial y} \gamma_1' + \frac{\gamma_2'}{\lvert \gamma_1 \rvert^{\alpha}} \lvert g_1 \rvert^{\alpha} \frac{\partial g_1}{\partial x} 
    \end{aligned}
    \end{equation}
a.e. in $[a,b]$. Assume that $\gamma$ is admissible. Thus $\gamma_1, \frac{\gamma_2'}{\lvert \gamma_1 \rvert^{\alpha}} \in L^1([a,b])$ and, by equations \eqref{derivatives of g} and Lemma \ref{wazny lemat}, expressions $\frac{\partial g_1}{\partial x}$ and $\lvert \gamma_1 \rvert^{\alpha}\frac{\partial g_1}{\partial y}$ are bounded in  $\gamma([a,b])$, cf. the end of the proof of Theorem \ref{important theorem}. Therefore, $\frac{d}{d t} g_1(\gamma(t)), \frac{\frac{d}{d t} g_2(\gamma(t))}{\lvert g_1 \rvert^{\alpha}} \in L^1([a,b])$. We now show that $g(\gamma)$ is absolutely continuous. It is clear by \eqref{derivatives of g} that $g_2$ is $C^1$, and therefore $g_2(\gamma)$ is absolutely continuous. Function $g_1(\gamma)$ is absolutely continuous since it is continuous, differentiable a.e. on $[a,b]$, the derivative $\frac{d}{dt}g_1(\gamma) \in L^1([a,b])$ and the image of the set of non--differentiability of $g_1$ is measure zero on the real line, i.e. $g_1(\{\gamma_1=0\})=\{0\}$ is measure zero. For more details c.f. with the proof of Banach--Zarecki Theorem in \cite{J Yeh}, pages 274--275. Thus $g(\gamma)$ is absolutely continuous and admissible.

Assume conversely, that $g(\gamma)$ is admissible. By virtue o Corollary \ref{basic properties of conformal} it holds that $g^{-1} \in HW^{1,\frac{2(\alpha+1)}{\alpha}}(\Omega',\Omega)$ is an orientation preserving conformal map. Thus $g^{-1}(g(\gamma))=\gamma$ is admissible.
\end{proof}
We remark that Theorem \ref{important theorem} allows to consider unbounded sets $\Omega, \Omega'$. In particular, $\Omega=\Omega'=G^2_{\alpha}$.
\begin{example}
    By Theorem \ref{important theorem} we know that orientation preserving and conformal mappings $g \in HW^{1,\frac{2(\alpha+1)}{\alpha}}_{\textrm{loc}}(G^2_{\alpha},G^2_{\alpha})$ are exactly the conjugates via Meyerson map of conformal mappings $\widetilde{g}=\widetilde{g}_1+i\widetilde{g}_2:\mathbb{C} \to \mathbb{C}$ such that $\widetilde{g}_1 = 0$ if and only if $Re(z)=0$, i.e. 
\begin{equation*}
    \widetilde{g}(z) = az+b, \qquad \widetilde{g}_1=0 \ \Leftrightarrow \ Re(z)=0,
\end{equation*}
for some $a,b \in \mathbb{C}$. Denote $a = a_1 + ia_2$, $b=b_1 + ib_2$ and $z = x+iy$. We calculate
\begin{equation*}
\begin{aligned}
    \widetilde{g}(z) &= (a_1 + ia_2)(x+iy) + (b_1 + ib_2) =  a_1x - a_2y + b_1 + i\widetilde{g}_2(z).  
\end{aligned}
\end{equation*}
Thus we obtain that the condition $\widetilde{g}_1 = 0$ if and only if $Re(z)=0$ is equivalent to $a_2=b_1=0$ and $a_1 \neq 0$.

Therefore, all the orientation preserving conformal mappings $g \in HW^{1,\frac{2(\alpha+1)}{\alpha}}_{\textrm{loc}}(G^2_{\alpha},G^2_{\alpha})$ are of the form $g = \varphi_{\alpha}^{-1} \circ \widetilde{g} \circ \varphi_{\alpha}$, where $\widetilde{g}(z) = Re(a)z + iIm(b)$, $Re(a) \neq 0$. Thus, by definition of $\varphi_{\alpha}$, we obtain that all the entire orientation preserving conformal maps in $G^2_{\alpha}$ are given by the following formula:
    \begin{equation*}
        g(x,y) = (\sign{(a)}\lvert a \rvert^{\frac{1}{\alpha+1}}x,ay+b)
    \end{equation*}
    for some $a \in \mathbb{R} \setminus \{0\}$ and $b \in \mathbb{R}$.
\end{example}
\begin{example}[Joukovski mapping on $G^2_{\alpha}$]\label{Joukovski map on Grushin}
    Recall the Joukovski mapping $w(z)= z +\frac{1}{z}$ for $z \in \mathbb{C} \setminus \{0\}$. It is conformal in a sense that $\frac{\partial w}{\partial \bar{z}}=0$ and $\frac{\partial w}{\partial z} \neq 0$ for $z \neq \pm 1$. It is easy to verify that $Re(w) = 0$ if and only if $Re(z)=0$. Let $U \subset \mathbb{C}$ be an open set such that $w|_U$ is a homeomorphism and $\pm 1 \notin U$. Then $\varphi_{\alpha}^{-1} \circ w \circ \varphi_{\alpha}$ is conformal on $\varphi_{\alpha}^{-1}(U)$. That means that the Joukovski map has its natural counterpart on the Grushin plane.
\end{example}

In the next observation we discuss the distortion of length of curves under conformal mappings in the Grushin plane.

\begin{prop}\label{length distortion estimate prop}
    Let $\Omega, \Omega' \subset G^2_{\alpha}$ be open, connected and assume that $g \in HW^{1,\frac{2(\alpha+1)}{\alpha}}(\Omega,\Omega')$ is an orientation preserving conformal map. Let $\gamma:[a,b] \to U \subset \Omega$ be an absolutely continuous curve in the Euclidean metric s.t. $\gamma' \neq 0$ a.e. in $(a,b)$ and $\gamma_1 \neq 0$ a.e. in $[a,b]$. Then there exist positive constants $C_1,C_2$ dependent only on $g$ and $U$ such that
\begin{equation*}
    C_1 l(\gamma) \le l(g(\gamma)) \le C_2 l(\gamma),
\end{equation*}
where $l$ stands for the length of a curve in $G^2_{\alpha}$.
\end{prop}

\begin{proof}
Let $\gamma:[a,b] \to U \subset \Omega$ be an absolutely continuous curve such that $\gamma(t) = (x(t),y(t))$ and denote
\begin{equation*}
    d_{\alpha} \gamma = \sqrt{\left(\frac{d x}{dt}\right)^2 + \frac{1}{\lvert x\rvert^{2\alpha}}\left(\frac{d y}{dt}\right)^2} \ dt.
\end{equation*}
Let $g = (g_1,g_2)$, $g \in HW^{1,\frac{2(\alpha+1)}{\alpha}}(\Omega,\Omega')$ be an orientation preserving conformal map. Then, for every open set $U \subset \subset \Omega$ the norm of the horizontal gradient $\lvert \nabla_H g_1 \rvert$ is bounded in $U \setminus \{x=0\}$, i.e. there exists a constant $C_1 = C_1(U,g)$ such that
\begin{equation}\label{bounded horizontal gradient}
    \lvert \nabla_H g_1(x,y) \rvert \le C_1 \textrm{ for every } (x,y) \in U \setminus \{x=0\}.
\end{equation}
Indeed, it follows from equations \eqref{derivatives of g} and Lemma \ref{wazny lemat}, cf. ending lines of the proof of Theorem \ref{important theorem}. Moreover, by direct computations, using equation $\overline{W}g=0$ in $R_g$
\begin{equation*}
\begin{aligned}
    d_{\alpha}g(\gamma) &= \sqrt{\left(\frac{d g_1}{dt}\right)^2 + \frac{1}{\lvert g_1\rvert^{2\alpha}}\left(\frac{d g_2}{dt}\right)^2} \ dt\\
    &= \sqrt{\left(\frac{\partial g_1}{\partial x} \frac{dx}{dt} + \frac{\partial g_1}{\partial y}\frac{dy}{dt}\right)^2 + \frac{1}{\lvert g_1 \rvert^{2\alpha}}\left( \frac{\lvert g_1 \rvert^{\alpha}}{\lvert x \rvert^{\alpha}} \frac{\partial g_1}{\partial x} \frac{dy}{dt} - \lvert x g_1 \rvert^{\alpha} \frac{\partial g_1}{\partial y} \frac{dx}{dt} \right)^2} dt\\
    &=\sqrt{\left(\frac{\partial g_1}{\partial x}\right)^2\left( \left( \frac{dx}{dt} \right)^2 + \frac{1}{\lvert x \rvert^2} \left( \frac{dy}{dt}\right)^2\right) + \lvert x \rvert^{2\alpha}\left(\frac{\partial g_1}{\partial y}\right)^2\left( \left( \frac{dx}{dt} \right)^2 + \frac{1}{\lvert x \rvert^{2\alpha}} \left( \frac{dy}{dt}\right)^2\right)} dt\\
    &=\sqrt{\left( \frac{\partial g_1}{\partial x}\right)^2 + \lvert x \rvert^{2\alpha} \left( \frac{\partial g_1}{\partial y} \right)^2} \sqrt{\left(\frac{d x}{dt}\right)^2 + \frac{1}{\lvert x\rvert^{2\alpha}}\left(\frac{d y}{dt}\right)^2} \ dt = \lvert \nabla_H g_1 (x,y) \rvert d_{\alpha}\gamma.
\end{aligned}
\end{equation*}
Therefore
\begin{equation*}
    d_{\alpha} \gamma =  d_{\alpha} (g \circ g^{-1})(\gamma) = \lvert \nabla_H g_1(g^{-1}(x,y)) \rvert \lvert \nabla_H (g^{-1})_1(x,y) \rvert d_{\alpha} \gamma.
\end{equation*}
Since $\gamma' \neq 0$ a.e. in $(a,b)$, we obtain 
\begin{equation}\label{equiv one}
\lvert \nabla_H g_1(g^{-1}(x,y)) \rvert \lvert \nabla_H (g^{-1})_1(x,y) \rvert \equiv 1    \quad \textrm{a.e. in } [a,b]
\end{equation}
and by \eqref{bounded horizontal gradient} and \eqref{equiv one} there exist positive constants $C_1, C_2$ dependent only on $g$ and $U$ such that
\begin{equation*}
   C_1 \le \lvert \nabla_H g_1(x,y) \rvert \le C_2 \quad \textrm{a.e. in } [a,b].
\end{equation*}
The assertion of the proposition easily follows. We note that by virtue of Corollary \ref{preserving admissible curves remark} there is no problem with definition of $g(\gamma)$.
\end{proof}

Let us recall the following theorem by Carath\'eodory.
\begin{theorem}\label{Caratheodory theorem}
    Let $\Omega_1$, $\Omega_2$ be bounded simply--connected Jordan domains in $\mathbb{C}$. If $f:\Omega_1 \to \Omega_2$ is a conformal mapping, then $f$ extends to a homeomorphism between $\overline{\Omega}_1$ and $\overline{\Omega}_2$
\end{theorem}
Using Theorem \ref{important theorem} we can rephrase the above theorem in the setting of the Grushin plane.
\begin{corollary}\label{caratheodory theorem on grushin}
    Let $\Omega_1$, $\Omega_2$ be bounded simply--connected Jordan domains in $G^2_{\alpha}$. If $g \in HW^{1,\frac{2(\alpha+1)}{\alpha}}_{\textrm{loc}}(\Omega,\Omega')$ is an orientation preserving conformal map, then $g$ extends to a homeomorphism $\hat{g} = (\hat{g}_1,\hat{g}_2)$ from $\overline{\Omega}_1$ to $\overline{\Omega}_2$. Moreover, $\hat{g}_1 = 0$ if and only if $x=0$.
\end{corollary}
\begin{proof}
    The proof is straightforward. Let $\Omega_1$ and $\Omega_2$ be bounded, simply--connected Jordan domains in $G^2_{\alpha}$. Let $\widetilde{g}:\varphi_{\alpha}(\Omega_1) \to \varphi_{\alpha}(\Omega_2)$ be defined as
    \begin{equation*}
        \widetilde{g} := \varphi_{\alpha} \circ g \circ \varphi_{\alpha}^{-1}.
    \end{equation*}
    By Theorem \ref{important theorem} mapping $\widetilde{g}$ is conformal and by properties of homeomorphisms $\varphi_{\alpha}(\Omega_1)$ and $\varphi_{\alpha}(\Omega_2)$ are bounded simply--connected domains in $\mathbb{C}$. Moreover, since $\varphi_{\alpha}$ is in fact homeomorphism on the whole plane, we have for homeomorphisms $h_i$, giving that $\partial \Omega_i$ are Jordan curves for $i=1,2$, that
    \begin{equation*}
        \varphi_{\alpha}(\partial \Omega_i) = \varphi_{\alpha}(h_i(S^1)) \qquad \textrm{for } i=1,2,
    \end{equation*}
    where $S^1$ denotes the unit circle. Thus, $\varphi_{\alpha}(\Omega_1)$ and $\varphi_{\alpha}(\Omega_2)$ are Jordan domains. From Theorem \ref{Caratheodory theorem} mapping $\widetilde{g}$ extends to a homeomorphism defined on $\overline{\varphi_{\alpha}(\Omega_1)}$ and similarly, we can extend $g$ to $\hat{g}$ defined on $\overline{\Omega}_1$. Observe that $\hat{g}(\overline{\Omega}_1)=\overline{\Omega}_2 $. Indeed, it is clear that $\hat{g}(\partial \Omega_1) \subset \partial \Omega_2$. If $\hat{g}(\partial \Omega_1) \neq \partial \Omega_2$ then it would contradict the fact that Jordan curve $\partial \Omega_2$ splits the plane into two connected components. It is sufficient to show that $\hat{g}_1 = 0$ if and only if $x=0$. By continuity of $\hat{g}$, we have that $\hat{g}_1 =0$ for $x=0$. Similarly, $g^{-1}:\Omega_2 \to \Omega_1$ can be extended to a homeomorphism defined on $\overline{\Omega}_2$. Moreover, by continuity of both extensions 
    \begin{equation*}
    \hat{g}^{-1} = \widehat{g^{-1}}.
    \end{equation*}
    Therefore, $\hat{g}_1 =0$ implies $x=0$. 
\end{proof}

\end{document}